\documentclass[12pt,twoside]{article}
\usepackage[all]{}
\usepackage {amssymb,latexsym,amsthm,amscd,amsmath}
\usepackage{url}
\newtheorem{theorem}{Theorem}[section]
\newtheorem{lemma}[theorem]{Lemma}
\newtheorem{corollary}[theorem]{Corollary}
\newtheorem{proposition}[theorem]{Proposition}

\newtheorem*{theorem*}{Theorem}
\theoremstyle{definition}

\numberwithin{equation}{subsection}

\newcommand{\ignore}[1]{}

\newcommand{\mynote}[1]{}
\begin{document}
\setcounter{section}{0}
\title{\bf Albert algebras and the Tits-Weiss conjecture}
\author{Maneesh Thakur }
\date{}
\maketitle
\begin{abstract}
\noindent
\it{We prove the Tits-Weiss conjecture for Albert division algebras over fields of arbitrary characteristic in the affirmative. The conjecture predicts that every norm similarity of an Albert division algebra is a product of a scalar homothety and $U$-operators. This conjecture is equivalent to the Kneser-Tits conjecture for simple, simply connected algebraic groups with Tits index $E^{78}_{8,2}$. We prove that a simple, simply connected algebraic group with Tits index $E_{8,2}^{78}$ or $E_{7,1}^{78}$, defined over a field of arbitrary characteristic, is $R$-trivial, in the sense of Manin, thereby proving the
  Kneser-Tits conjecture for such groups. The Tits-Weiss conjecture follows as a consequence.}
\end{abstract} 

\section{\bf Introduction}
The primary goal of this paper is to prove the Tits-Weiss conjecture for Albert division algebras (i.e. exceptional simple Jordan algebras) over a field of arbitrary characteristic. 
A recent conjecture of J. Tits and R. M. Weiss predicts that the structure group (i.e. the group of norm similarities) of an Albert algebra $A$ over a field $k$ is generated by scalar homotheties and $U$-operators of $A$ (see \cite{TW}, 37.41, 37.42 and page 418). This conjecture is equivalent to the Kneser-Tits conjecture for groups with Tits index $E_{8,2}^{78}$ (see \cite{ACP} Appendix).

Let $\text{\bf G}$ be a connected, simple, simply connected algebraic group, defined and isotropic over a field $k$. The \emph{Kneser-Tits conjecture} predicts $\text{\bf G}(k)/\text{\bf G}(k)^{\dagger}=\{1\}$, where $\text{\bf G}(k)^{\dagger}$ is the subgroup of $\text{\bf G}(k)$, the group of $k$-rational points of $\text{\bf G}$, generated by the $k$-rational points of the unipotent radicals of $k$-parabolics of $\text{\bf G}$ (see \cite{G}, \cite{PRap}). The Kneser-Tits conjecture is false in general (see \cite{G}, \cite{PRap}), however one is interested in computing the quotient
$W(\text{\bf G},k):=\text{\bf G}(k)/\text{\bf G}(k)^{\dagger}$, called the \emph{Whitehead group} of $\text{\bf G}$. The Kneser-Tits conjecture was proved to hold over local fields by Platonov and the conjecture is now known to hold over global fields as well, by works of several authors (see \cite{G}, \cite{PRap}).

For an exhaustive survey on the Kneser-Tits problem, we refer to (\cite{G}, \cite{PRap}, \cite{T0}, \cite{T2}) and to (\cite{P-R1}) for reduction to the relative rank-$1$ case. The Kneser-Tits conjecture for groups with index $E_{8,2}^{66}$ was settled in (\cite{P-T-W}).

By a recent result of P. Gille (\cite{G}, Thm. 7.2), the validity of Kneser-Tits conjecture for an algebraic group $\text{\bf G}$ satisfying the Kneser-Tits hypothesis is equivalent to vanishing of the group of $R$-equivalent points $\text{\bf G}(k)/R$ for $\text{\bf G}$. This facilitates our strategy to prove the Tits-Weiss conjecture. In the paper we prove that algebraic groups with Tits indices $E_{8,2}^{78}$ and $E_{7,1}^{78}$, over fields of arbitrary characteristic, are $R$-trivial, in the sense of Manin (see \cite{M}).   

Let $k$ be a field of arbitrary characteristic and $\text{\bf G}$ be a simple, simply connected algebraic group, defined over a field $k$, having Tits index $E_{8,2}^{78}$ or $E_{7,1}^{78}$ (see \cite{T3} for the notation). We prove that $\text{\bf G}$ is $R$-trivial, i.e. $\text{\bf G}(L)/R=\{1\}$ for any field extension $L$ of $k$, in the sense of Manin (\cite{M}).

Groups with these indices exist over a field $k$ if there exist central associative division algebras of degree $3$ over $k$, whose reduced norm maps are not surjective or if there are central associative division algebras of degree $3$ with unitary involutions over a quadratic field extension $K$ of $k$, with non-surjective reduced norm map (see \cite{T1}, 3.2 and 3.3.1). For example, such groups exist over $\mathbb{Q}(t)$, the function field in one variable over $\mathbb{Q}$, while they do not exist over local fields, global fields and (algebraic extensions of) finite fields.

The anisotropic kernels of  groups $\text{\bf G}$ having any of the aforementioned indices are classified by isotopy classes of Albert division algebras over $k$. In fact, these are isomorphic to the structure groups of Albert division algebras over $k$ (see \cite{T1}). The Tits buildings associated to
groups of type $E_{8,2}^{78}$ are Moufang hexagons coming from hexagonal systems of type $27/K$ where $K$ is a quadratic field extension of $k$ or $K=k$ (\cite{TH}, \cite{TW}). Simple groups of type $E_{7,1}^{78}$ similarly correspond to Moufang sets $M(A)$ where $A$ is an Albert division algebra over $k$ (see \cite{TS}).

We prove the Tits-Weiss conjecture via $R$-triviality of algebraic groups with index $E_{8,2}^{78}$, in the affirmative. This is achieved by proving that the structure group of the Albert division algebra corresponding to the anisotropic kernel of such a group, is $R$-trivial. 

There are two rational constructions of Albert algebras over a field $k$ due to J. Tits, known as \emph{first Tits construction} and the \emph{second Tits construction}. In the two papers (\cite{Th-1}, \cite{Th-2}) the Tits-Weiss conjecture for first Tits construction Albert division algebras and also for reduced Albert algebras was settled in the affirmative. The present paper is a sequel to (\cite{Th-2}) and completes the investigation on the Tits-Weiss conjecture.

\vskip1mm
\noindent
{\bf Remark.} This paper appeared on the arXiv on $9$th November-2019. About three weeks later, on 28th November-2019, a paper by Alsaody et al (see \cite{ACP}) appeared on the arXiv, proving the Tits-Weiss conjecture. Their methods are largely cohomological, while our techniques are explicit and use computations on the subgroups of the structure group of an Albert division algebra. 

\section{\bf Preliminaries}
In this section, we give a brief introduction to various notions central to the paper, as well as set up notation. We refer the reader to (\cite{P2}, \cite{PR5}, \cite{SV}) and references therein, for basic material on Albert algebras.
\vskip1mm
\noindent
{\bf (a) Cubic norm structures, Tits process, Albert algebras.}
In this paper we encounter quadratic Jordan algebras of degree $3$ over a field. These can be neatly described and constructed via \emph{cubic norm structures}. Recall that a cubic norm structure on a vector space $V$ over a field $k$ is a triple $(N,\#, c)$, where the \emph{norm} $N:V\rightarrow k$ is a cubic form, the \emph{adjoint} $\#:V\rightarrow V$ is a quadratic map and the \emph{identity} $c\in V$ is a base point and $N(c)=1$. To avoid repetition, we refer the reader to (\cite{Th-2}, \S2) for more details, notation and exposition on cubic norm structures and their connection with quadratic Jordan algebras that we use in this paper.

Let $K$ be a quadratic \'{e}tale extension of $k$ and $(B,\sigma)$ be a separable associative algebra of generically algebraic degree $3$ over $k$, with an involution $\sigma$ of the second kind. Taking the unit element $1$ of $B$ as identity 
element, the restriction of the norm $N_B$ as norm and the restriction of the adjoint map of $B$ to $(B,\sigma)_+:=\{b\in B|\sigma(b)=b\}$ as adjoint, one obtains a cubic norm structure on $(B,\sigma)_+$. We therefore get a (quadratic) Jordan algebra structure on $(B,\sigma)_+$.

The case when $K=k\times k$, we necessarily have $(B,\sigma)=(E\times E^{\circ},\epsilon)$, where $E$ is a generically algebraic degree $3$ separable associative algebra over $k$, $E^{\circ}$ is its opposite algebra and $\epsilon$ is the switch involution on $E\times E^{\circ}$. The cubic norm structure in this case is identified with $(N_E, \#, 1_E)$ (see \cite{Th-2} for notation) and the associated quadratic Jordan algebra is denoted by $E_+$.
\vskip1mm
\noindent
{\bf Tits process.} 
Let $K$ be a quadratic \'{e}tale extension of $k$ and $B$ be a separable associative algebra of generically algebraic degree $3$ over $K$ with an involution $\sigma$ over $k$. Let $u\in (B,\sigma)_+$ be a unit with $N_B(u)=\mu\overline{\mu}$ for some $\mu\in K^{\times}$, where $a\mapsto \overline{a}$ is the nontrivial $k$-automorphism of $K$.  

Consider the cubic norm structure on the $k$-vector space $(B,\sigma)_+\oplus B$, given by the norm, adjoint and identity respectively as follows:
$$N((b,x)):=N_B(b)+T_K(\mu N_B(x))-T_B(bxu\sigma(x)),$$
$$(b,x)^{\#}:=(b^{\#}-xu\sigma(x), \overline{\mu}\sigma(x)^{\#}u^{-1}-bx),~1:=(1_B,0).$$ 

The associated (quadratic) Jordan algebra is denoted by $J(B,\sigma,u,\mu)$. Here $(B,\sigma)_+$ can be identified with a subalgebra of $J(B,\sigma, u,\mu)$ through the first summand. Moreover, $J(B,\sigma,u,\mu)$ is a division algebra if and only if $\mu$ is not a norm from $B$ (see \cite{PR7}, 5.2). This construction is called the \emph{Tits process} arising from the parameters $(B,\sigma),u$ and $\mu$. The pair $(u,\mu)$ as above is referred to as an \emph{admissible pair}. If $S:=(B,\sigma)_+$, then we sometimes write $J(S, u, \mu)$ for $J(B,\sigma, u, \mu)$ in \S3 of this paper.

If we start the Tits process with $(B,\sigma)$, a degree $3$ central simple algebra with a unitary involution $\sigma$ over its center $K$, we obtain a quadratic Jordan algebra $J(B,\sigma, u, \mu)$, called an \emph{Albert algebra}. All Albert algebras arise this way. Tits process in the context of Albert algebras are called Tits constructions. 
\vskip1mm
\noindent
{\bf (b) Isotopes.} If $v$ is an invertible element of a cubic norm structure $J(N,\#,c)$, the \emph{$v$-isotope} of $J(N, \#, c)$ is the cubic norm structure 
$$J(N, \#, c)^{(v)}:=J(N^{(v)}, \#^{(v)},c^{(v)}),~\text{with}~N^{(v)}(x)=N(v)N(x),$$
$$~x^{\#^{(v)}}=N(v)U_v^{-1}(x^{\#}),~c^{(v)}=v^{-1}=N(v)^{-1}v^{\#}.$$
We also need explicit description of isotopes of a Tits process. Let $L$ be a cubic \'{e}tale algebra and $K$ a quadratic \'{e}tale algebra over a field $k$. Let $*:LK\rightarrow LK$ be the nontrivial automorphism of $K$, extended $L$-linearly to an automorphism of $LK$. Let $(u,\mu)\in L^{\times}\times K^{\times}$ be an admissible pair. Consider the Tits process $J(LK, *, u, \mu)$. Then, for $v\in L^{\times}$, the $v$-isotope $J(LK, *, u, \mu)^{(v)}$ is isomorphic to $J(LK, *, uv^{\#}, N(v)\mu)$ via the map $ J(LK, *, u, \mu)^{(v)}\rightarrow J(LK, *, uv^{\#}, N(v)\mu)$ given by 
$$(l,x)\mapsto (lv,x)~\text{for}~l\in L^{(v)}=L,~x\in LK.$$
Similarly, for an Albert algebra $A=J(B,\sigma, u, \mu)$ as in ({\bf a}), for any invertible $v\in (B,\sigma)_+$, the $v$-isotope $A^{(v)}=J(B,\sigma, u, \mu)^{(v)}$ is isomorphic to the Albert algebra $J(B, \sigma_v, uv^{\#}, N(v)\mu)$ via the map $(b,x)\mapsto (vb,x)$ for $b\in (B,\sigma)_+^{(v)}, ~x\in B$ (see \cite{PR7}, Prop. 3.9), here $\sigma_v:B\rightarrow B$ is the involution
$\sigma_v(x)=v\sigma(x)v^{-1},~x\in B$.   

\vskip1mm
\noindent
{\bf (c) Isomorphisms, norm similarities, $U$-operators.} Let $J_1$ and $J_2$ be cubic norm structures over a field $k$, with norms $N_1$ and $N_2$ respectively, and $1$ denote the unit element in both structures. A linear bijection $f:J_1\rightarrow J_2$ is a \emph{norm similarity} if there exists $\nu(f)\in k^{\times}$ such that $N_2(f(x))=\nu(f)N_1(x)$ for all $x\in J_1$. 

A norm similarity is an isomorphism of the associated Jordan algebras if and only if it maps $1$ to $1$ (see \cite{McK}, \S4, page 507). In particular, a norm similarity of a cubic norm structure to itself is an automorphism if and only if it fixes the unit element $1$. 

For $k$ infinite, norm similarities are same as isotopies for degree $3$ Jordan algebras (see \cite{J1}, Chap. VI, Thm. 6, Thm. 7). For a Jordan algebra $J$ of degree $3$ (equivalently, a cubic norm structure) given as a Tits process (see {\bf (a)}), the operator $U_a$, for $a\in J$ invertible, is defined by $U_x(y):=T(x,y)x-x^{\#}\times y,~y\in J$, where $T(-,-)$ is the trace bilinear form on $J$, is a norm similarity. For any $x\in J$, we have $N(U_a(x))=N(a)^2N(x)$. We also have, for $x\in J$, $N(x^{\#})=N(x)^2$.

For a Jordan algebra $J$ of degree $3$ over a field $k$ and a norm similiarity $\psi:J\rightarrow J$, $a:=\psi(1)$ is invertible in $J$ and $a$ is the identity element of the isotope $J^{(v^{-1})}$. Moreover $\psi:J\rightarrow J^{(a^{-1})}$ is an isomorphism of Jordan algebras (see \cite{J1}, proof of (2), Thm. 7). 

\vskip1mm
\noindent
{\bf (d) Subalgebras of Albert division algebras.} If $A$ is an Albert division algebra over a field $k$, then any proper subalgebra of $A$ is either $k$ or a cubic subfield of $A$ or of the form $(B,\sigma)_+$ for a degree $3$ central simple algebra $B$ with an involution $\sigma$ of the second kind over its center $K$, a quadratic \'{e}tale extension of $k$ (see \cite{J1}, Chap. IX, \S 12, \cite{PR6}, Thm. 1.1). 
\vskip1mm
\noindent
{\bf (e) $R$-triviality, Whitehead groups, Kneser-Tits problem.}
Let $X$ be an irreducible variety over a field $k$ with $X(k)$ nonempty. Define $x,y\in X(k)$ to be $R$-equivalent 
if there exist $x_0=x, x_1,\cdots, x_n=y$ points in $X(k)$ together with rational maps $f_i:\mathbb{A}_k^1\rightarrow X,~1\leq i\leq n$ defined over $k$, $f_i$ regular at $0$ and $1$, with $f_i(0)=x_{i-1},~f_i(1)=x_i$ (see \cite{M}). 

Let $\text{\bf G}$ be a connected algebraic group defined over $k$. The $R$-equivalence is an equivalence relation on $\text{\bf G}(k)$. 
Let $R\text{\bf G}(k):=\{g\in\text{\bf G}(k)|g~\text{is~$R$-equivalent~to~$1\in\text{\bf G}(k)$}\}$. Then $R\text{\bf G}(k)$ is a normal subgroup of $\text{\bf G}(k)$ and the set $\text{\bf G}(k)/R$ of $R$-equivalence classes in $\text{\bf G}(k)$ can be identified with $\text{\bf G}(k)/R\text{\bf G}(k)$ and hence has a natural group structure. We identify $\text{\bf G}(k)/R$ with the group $\text{\bf G}(k)/R\text{\bf G}(k)$. 

We define $\text{\bf G}$ to be \emph{$R$-trivial} if $\text{\bf G}(L)/R=\{1\}$ for all field extensions $L$ of $k$. A variety $X$ defined over $k$ is $k$-\emph{rational} if $X$ is birationally isomorphic, over $k$, to an affine space. It is known (see \cite{Vos}, Chap. 6, Prop. 2) that if $\text{\bf G}$ is $k$-rational then $\text{\bf G}$ is $R$-trivial.  

Let $\text{\bf G}$ be a connected reductive group, defined and isotropic over $k$. Let $\text{\bf G}(k)^{\dagger}$ be subgroup of $\text{\bf G}(k)$ generated by the $k$-rational points of the unipotent radicals of the parabolic $k$-subgroups of $\text{\bf G}$. Then $\text{\bf G}(k)^{\dagger}$ is normal in $\text{\bf G}(k)$ and the group $W(k,\text{\bf G})=\text{\bf G}(k)/\text{\bf G}(k)^{\dagger}$ is called the \emph{Whitehead group} of $\text{\bf G}$. 

It is known that if $\text{\bf G}$ is a semisimple, simply connected and absolutely almost simple group, defined and isotropic over $k$, then $W(k,\text{\bf G})\cong \text{\bf G}(k)/R$ (\cite {G}, Thm. 7.2). The Kneser-Tits problem asks if for such a group $\text{\bf G}$, $W(k,\text{\bf G})=\{1\}$.

We end this section by stating the Main theorem of the paper and its consequences:
\vskip1mm
\noindent
    {\bf Theorem.} (see Theorem. \ref{main}) \emph{ Let $A$ be an Albert division algebra over a field $k$ of arbitrary characteristic. Then {\bf Str}$(A)$ is $R$-trivial.}
    \vskip1mm
    The Kneser-Tits conjecture for groups with Tits index $E_{7,1}^{78}$ or $E_{8,2}^{78}$, as well as the Tits-Weiss conjecture for Albert division algebras follow from this, see Corollary \ref{KnTits-1}, Corollary \ref{KnTits-2} and Corollary \ref{TW}.

\section{\bf $R$-triviality results} In this section, we prove some results on $R$-triviality of various groups, instrumental in   
proving the Main theorem. We begin by setting up some notation.

For a ring $R$, we denote the set of units in $R$ by $R^{\times}$. For a $k$-algebra $X$ and a field extension $L$ of $k$, We will denote by $X_L$ the $L$-algebra $X\otimes_kL$. More generally, for any object $X$ defined over $k$, $X_L$ will denote its base change to $L$.

For a ring $R$ and $a\in R$, we will denote the right homothety by $a$ on $R$ by $R_a$. Hence $R_a:R\rightarrow R$ and $R_a(x)=xa$ for all $x\in R$. 

Let $k$ be an infinite field of arbitrary characteristic and $\overline{k}$ be an algebraic closure of $k$. Let $A$ be an Albert algebra over $k$ with norm $N$ (see \S2 {\bf(a)}). It is well known that the full group of automorphisms $\text{\bf Aut}(A):=\text{Aut}(A_{\overline{k}})$ is a simple, simply connected algebraic group of type $F_4$ defined over $k$ and all simple groups of type $F_4$ defined over $k$ arise this way (see e.g. \cite{SV}, \cite{Spr}). 

We will denote the group of $k$-rational points of $\text{\bf Aut}(A)$ by $\text{Aut}(A)$. It is known that $A$ is a division algebra if and only if the norm form $N$ of $A$ is anisotropic (see \cite{Spr}, Thm. 17.6.5). Albert algebras whose norm form is isotropic over $k$ are called \emph{reduced}.

Let $A$ be an Albert algebra over $k$ and $S\subset A$ a subalgebra. By $\text{\bf Aut}(A/S)$ we denote the (closed) subgroup of $\text{\bf Aut}(A)$ consisting of automorphisms of $A$ which fix $S$ pointwise and $\text{\bf Aut}(A,S)$ denotes the closed subgroup of automorphisms of $A$ leaving $S$ invariant. We denote $\text{Aut}(A,S):= \text{\bf Aut}(A,S)(k)$ and $\text{Aut}(A/S):=\text{\bf Aut}(A/S)(k)$.

The \emph{structure group} of $A$ is the full group $\text{\bf Str}(A):=\text{Str}(A_{\overline{k}})$ of norm similarities of $N$, is a connected reductive group over $k$, of type $E_6$ (see \cite{SV}). We denote by $\text{Str}(A):=\text{\bf Str}(A)(k)$, the group of $k$-rational 
points of $\text{\bf Str}(A)$. The automorphism group $\text{\bf Aut}(A)$ is the stabilizer of $1\in A$ in $\text{\bf Str}(A)$ (see \cite{McK}, page 507). In particular, $\text{\bf Aut}(A)\subset\text{\bf Str}(A)$. The commutator subgroup $\text{\bf Isom}(A)$ of $\text{\bf Str}(A)$ is the full group of isometries of $N$ and is a simple, simply connected group of type $E_6$, anisotropic over $k$ if and only if $N$ is anisotropic, if and only if $A$ is a division algebra (see \cite{SV}, \cite{Spr}).

The $U$-operators $U_a$, for $a\in A$ invertible, are norm similarities (see \S2 {\bf(c)}) and generate a normal subgroup of $\text{Str}(A)$, called the \emph{inner structure group} of $A$, which we will denote by $\text{Instr}(A)$.

Similarly, $\text{\bf Str}(A,S)$ denotes the closed subgroup of $\text{\bf Str}(A)$ whose elements leave the subalgebra $S$ invariant. Let $\text{Str}(A,S) := \text{\bf Str}(A,S)(k)$. Denote by $\text{\bf Str}(A/S)$ the closed subgroup of $\text{\bf Str}(A)$ leaving the subalgebra $S$ pointwise fixed and $\text{Str}(A/S) :=\text{\bf Str}(A/S)(k)$. 
Since any norm similarity of $A$ fixing the identity element $1$ of $A$ is necessarily an automorphism of $A$, it follows that 
${\bf Str}$(A/S)={\bf Aut}$(A/S)$.
\vskip1mm
\noindent
{\bf Remark.} It is known that for an Albert division algebra $A$ and a $9$-dimensional subalgebra $S\subset A$, the subgroup $\text{\bf Aut}(A/S)$ of $\text{\bf Aut}(A)$ described above, is simply connected, simple of type $A_2$, defined over $k$ (see Thm. 8, Chap. IX, \cite{J1}, \cite{FP}, \cite{KMRT}, \S 39.B). Since these are rank-$2$ groups, they are are rational and hence are $R$-trivial. We will need this fact.
\vskip1mm
\noindent
{\bf Fixed points.} For a group $H$ acting on a set $X$, we denote the set of fixed points of $H$ in $X$ by $X^H$. For an element $\phi\in H$, $X^{\phi}$ denotes the set of points in $X$ fixed by $\phi$. When $X$ is an algebra and $H\subset\text{Aut}(X)$, the set $X^H$ is a subalgebra of $X$. We will use these notations and facts without any mention.  
\smallskip
\noindent
We recall a few results from (\cite{Th-2}) for ready referencing. 
\begin{lemma}[Thm. 4.1, \cite{Th-2}]\label{fixedpoint} Let $A$ be an Albert division algebra over $k$ and $\phi\in\text{Aut}(A)$ an automorphism of $A$. Then $\phi$ fixes a cubic subfield of $A$ pointwise.
\end{lemma} 
\begin{proposition}[Thm. 5.1, \cite{Th-2}]\label{Str-R-triv} Let $A$ be an Albert (division) algebra. Let $S\subset A$ be a $9$-dimensional subalgebra. Then, with the notations as above, $\text{\bf Str}(A,S)$ is $R$-trivial. 
\end{proposition}
\begin{corollary}[Cor. 5.1, \cite{Th-2}]\label{Aut-R-triv-S} Let $A$ be an Albert (division) algebra and $S$ a $9$-dimensional subalgebra of $A$. Then $\text{Aut}(A,S)\subset R(\text{\bf Str}(A)(k))$.
\end{corollary}

\noindent
We now prove some auxiliary results, needed in the proof of the Main theorem. 
  \begin{proposition}\label{invariant-cubic} Let $A$ be an Albert division algebra over a field $k$ and $\psi\in\text{Str}(A)$. Then there exists a cubic subfield $M\subset A$ and $\theta\in R\text{\bf Str}(A)(k)$ such that
    $\theta\psi\in\text{Str}(A,M)$.
  \end{proposition}
  \begin{proof} Let $a:=\psi(1)$. If $a\in k$ then $R_{a^{-1}}\psi(1)=1$. Thus, by (\cite{McK}, \S4, page 507), $\phi:=R_{a^{-1}}\psi\in \text{Aut}(A)$ and by Lemma \ref{fixedpoint}, $\phi$ fixes a cubic subfield $M$ of $A$, hence $\psi$ stabilizes $M$ and the result follows in this case. So we may assume $a\notin k$. Then $L:=k(a)$, the subalgebra generated by $a$, is a cubic subfield of $A$. 
    \vskip1mm
    \noindent
        {\bf Claim.} There exists $\phi\in\text{ Aut}(A)$ such that $A^{\phi}=L$.
\vskip1mm

        To see this, we argue as follows. We consider the subgroup $\text{\bf Aut}(A/L)$ of $\text{\bf Aut}(A)$. One knows that $\text{\bf Aut}(A/L)$ is a simply connected, simple group of type $D_4$ defined over $k$ (see \S 17.9, \cite{Spr} or \cite{SV}). Let $T\subset\text{\bf Aut}(A/L)$ be any maximal $k$-torus. Then $T$ is a rank-$4$ torus. Let $\phi\in T(k)$ be a generator for $T$ (see \cite{B}, Prop. 8.8, Remark, or Prop. 1.6.4, \cite{T4}). Then by Zariski density, it follows that $A^T= A^{\phi}$. 

Since $T\subset\text{\bf Aut}(A/L)$, we have $L\subset A^T$. If $L\neq A^T$, then writing $S:=A^T$, by (\S2 {\bf(d)}), $S$ must be a $9$-dimensional subalgebra of $A$. Hence $T\subset\text{\bf Aut}(A/S)$. But by the Remark above, $\text{\bf Aut}(A/S)$ is a rank-$2$ group, and $T$ is rank-$4$, hence $T\not\subset\text{\bf Aut}(A/S)$, a contradiction. Therefore $L=A^T=A^{\phi}$. In particular $\phi(a)=a$. 

        We have
        $$\psi^{-1}\phi\psi(1)=\psi^{-1}\phi(a)=\psi^{-1}(a)=1.$$
        Hence $\psi^{-1}\phi\psi\in\text{Aut}(A)$. We again compute, for $x\in A$,
        $$ x\in A^{\psi^{-1}\phi\psi}\iff \psi^{-1}\phi\psi(x)=x\iff \phi\psi(x)=\psi(x)$$
        $$\iff \psi(x)\in A^{\phi}=L=k(a)\iff x\in \psi^{-1}(L).$$
        Therefore $M:=A^{\psi^{-1}\phi\psi}=\psi^{-1}(L)$ is a subfield of $A$ and $\psi:M=\psi^{-1}(L)\rightarrow L$. Let $J:=<M, L>$, the subalgebra generated by $M$ and $L$ in $A$.
        If $\psi^{-1}(L)=L=M$, then the assertion of the theorem follows by taking $M=L$ and $\theta=1$.

        Hence we may assume $M\neq L$ and thus $J$ is a $9$-dimensional subalgebra of $A$ (see \S2 {\bf(d)}).
        Following (\cite{GP}), we consider the isotopic embeddings $\iota':M\rightarrow J$ given by $\iota'(m)=\psi(m)\in L\subset J$ and $\iota:M\hookrightarrow J$ is the inclusion. Then, by (\cite{GP}, Thm. 5.2.7), there exists $\theta\in \text{Str}(J)$
        such that $\theta\iota'=\iota R_w$ for some element $w\in M$ with $N_M(w)=1$. Then we have, for any $x\in M$,
        $ \theta\iota'(x)=\iota R_w(x)$. Hence $\theta(\psi(x))=wx\in M~\forall x\in M$. Therefore $\theta\psi(M)=M$.

        By (Thm. 4.4, \cite{Th-2}), we can extend $\theta$ to a norm similarity of $A$,
        which we continue to denote by $\theta$. Then $\theta\in \text{Str}(A,J)$ and hence $\theta\in R\text{\bf Str}(A)(k)$ by 
Proposition \ref{Str-R-triv}. So the result follows. 
  \end{proof}
  \noindent In the following, let $M\subset A$ be a fixed cubic subfield and $\mathcal{S}_M$ denote the set of all $9$-dimensional subalgebras of $A$ which contain $M$. We have,
  \begin{proposition}\label{action} With the notation as above, the group $\text{Str}(A,M)$ acts on $\mathcal{S}_M$.
  \end{proposition}
  \begin{proof} Let $\psi\in\text{Str}(A,M)$ and $S\in\mathcal{S}_M$ be arbitrary. Fix $\phi\in\text{Aut}(A/S),~\phi\neq 1$. Then $S\subset A^{\phi}$. Since $\phi\neq 1$ and $\text{dim}(S)=9$, by (\S2 {\bf(d)}), it follows that $A^{\phi}=S$. Since $\psi^{-1}(1)\in M\subset S$ and $\phi$ fixes $S$ pointwise, we have
    $$\psi\phi\psi^{-1}(1)=\psi(\psi^{-1}(1))=1,$$
    Hence, by (\S2 {\bf(c)}), $\psi\phi\psi^{-1}\in\text{Aut}(A)$. Moreover, for any $x\in S$, we have
    $$\psi\phi\psi^{-1}(\psi(x))=\psi\phi(x)=\psi(x).$$
   Hence $A^{\psi\phi\psi^{-1}}=\psi(S)$ and since $\psi\phi\psi^{-1}$ is an automorphism of $A$,  $\psi(S)$ is a $9$-dimensional subalgebra of $A$ and $\psi(S)\supset\psi(M)=M$. Hence $\psi(S)\in\mathcal{S}_M$ and we have proved the theorem. 
 \end{proof}
 \noindent
     {\bf Notation:} Let $A$ be an Albert division algebra and $M\subset A$ a cubic subfield. Then for $\psi\in\text{Str}(A,M)$ and $S\in\mathcal{S}_M$, by the above, $\psi(S)\in\mathcal{S}_M$, i.e. $\psi(S)$ is a subalgebra of $A$ and conatins $M$ as a subfield. In particular, the identity elements of $\psi(S)$ and $A$ coincide. On the other hand, if $a:=\psi(1)\in M^{\times}$, then $\psi: A\rightarrow A^{(a^{-1})}$ is an isomorphism and $\psi(S)$ is therefore a subalgebra of $A^{(a^{-1})}$ containing $\psi(M)=M$ and both $A^{(a^{-1})}$ and $\psi(S)$ have $a=\psi(1)$ as identity element. To distinguish $\psi(S)\subset A$ from $\psi(S)\subset A^{(a^{-1})}$, we will write $[\psi(S)]$ for $\psi(S)$ when treated as a subalgebra of $A$.

 \vskip1mm
 \noindent
{\bf Cyclic Albert algebras.} We now prove some results for cyclic Albert division algebras, crucial to the proof of the Main theorem.
 An Albert algebra $A$ over a field $k$ is said to be \emph{cyclic} if $A$ contains a cubic cyclic field extension of $k$ as a subalgebra. Albert division algebras that arise from the first Tits construction are cyclic (see \cite{PR2}). 

\noindent
Let $L/k$ be a cyclic cubic field extension and $K/k$ a quadratic separable field extension. Let $\rho$ be a generator for $Gal(L/k)$ also let $\rho$ denote the $K$-linear field automorphism of $LK$ extending $\rho\in Gal(L/k)$.
Let $x\mapsto \overline{x}$ denote the nontrivial field automorphism of $K/k$. We now prove,
\begin{lemma}\label{tits-pro-skol} Let $J=J(LK, *, 1, \nu)$ be a Tits process, where $*$ denotes the nontrivial automorphism of $LK/L$, that fixes $L$ pointwise. Let $\rho$ be a generator of $Gal(L/k)$. Then $\rho$ extends to an automorphism of $J$.
\end{lemma}
\begin{proof} We have $J=L\oplus LK$ as $k$-vector spaces. Define $\tilde{\rho}:J\rightarrow J$ by
  $$\widetilde{\rho}((l,x))=(\rho(l), \rho(x)),$$
  where we also denote by $\rho$ the $K$-linear extension of $\rho$ to an automorphism of $LK$. To prove $\widetilde{\rho}$ is an automorphism of $J$, by (\S2 {\bf(c)}), it suffices to check it is bijective, $k$-linear and it preserves the identity and norm of $J$. It is obvious that $\widetilde{\rho}$ is $k$-linear, bijective and $\widetilde{\rho}((1,0))=(1,0)$. We check below that the norms are preserved. We have, using the fact that $*$ and $\rho$ centralize each other in $Gal(LK/k)$, 
  $$N_J(\widetilde{\rho}((l,x)))=N_J((\rho(l), \rho(x)))~~~~~~~~~~~~~~~~~~~~~~~~~~~~~~~~~~~~~~~~~~~~~~~~~~~~~~~$$
$$~~~~~~~~~~=N_L(\rho(l))+T_{K/k}(\nu N_{LK}(\rho(x)))-T_{LK}(\rho(l)\rho(x)\rho(x)^*)$$
$$=N_L(l)+T_{K/k}(\nu N_{LK}(x))-T_{LK}(\rho(lxx^*))~~~~~$$
$$~~~~~~=N_L(l)+T_K(\nu N_{LK}(x))-T_{LK}(lxx^*)=N_J((l,x)).$$
\end{proof}
\begin{lemma}\label{special-subalg} Let $A$ be an Albert division algebra over a field $k$ and $L\subset A$ be a separable cubic subfield. Then there exists $v\in L^{\times}$ such that the isotope $A^{(v)}$ contains a subalgebra isomorphic to a Tits process algebra $J(LK,*, 1,\nu)$ for a suitable quadratic \'{e}tale extension $K/k$ and a suitable norm $1$ element $\nu\in K^{\times}$. 
\end{lemma}
\begin{proof} Let $(B,\sigma)_+$ be a $9$-dimensional subalgebra of $A$ such that $L\subset (B,\sigma)_+\subset A$. The inclusion $(B,\sigma)_+\hookrightarrow A$ extends to an isomorphism $J(B,\sigma, u, \mu)\cong A$ for suitable admissible pair $(u,\mu)$ (see \cite{McK-2}, Thm. 9). Hence, we may assume $A=J(B,\sigma, u, \mu)$. 

  If the center $K:=Z(B)$ of $B$ is split, then $(B, \sigma)=(D\times D^{\circ}, \epsilon)$ for a degree $3$ central division algebra $D$ over $k$, $\epsilon$ is the switch involution on $B$. In this case $A$ is a first Tits construction and $D_+\supset L$. Hence $S=D_+=J(L,\nu)$ for a suitable $\nu\in k^{\times}$ is a subalgebra of $A$ and can be identified with the Tits process
  $J(LK,*,(1,1),(\nu,\nu^{-1}))$, where $*:LK\rightarrow LK$ is the switch involution.

  We may therefore assume that $K$ is a separable quadratic field extension of $k$. Since $L\subset (B,\sigma)_+$, by (\cite{PT}, 1.6, \cite{P}, 6.4), there exists an isomorphism 
$J(LK, *,v, \nu)\cong (B,\sigma)_+$ for some $v\in L^{\times}$ with $N_L(v)=1$, $\nu\in K$ with $N_K(\nu)=1$, which extends the inclusion $L\hookrightarrow (B,\sigma)_+$. This induces an isomorphism 
$$A=J(B,\sigma, u, \mu)\cong J(S_1, u_1, \mu),$$
where $S_1:=J(LK,*, v, \nu)$ and $u_1$ is the image of $u\in (B,\sigma)_+$ in $S_1$ under the above isomorphism.

This isomorphism is identity on $L$ and maps $v$ to $v$, hence we have an induced isomorphism $A^{(v)}\cong A_1^{(v)}$, where $A_1:=J(S_1,u_1,\mu)$. 
By (\S2 {\bf(b)}), we have 
$$S_1^{(v)}=J(LK,*, v, \nu)^{(v)}\cong J(LK, *, 1, \nu),$$ 
via the isomorphism given by $(l,x)\mapsto (lv, x)$ for $l\in L,~x\in LK$. This isomorphism leaves $L$ invariant and $S_1^{(v)}\subset A_1^{(v)}$. Pulling back $S_1^{(v)}$ to $A^{(v)}$ by the above isomorphism yields a subalgebra of $A^{(v)}$ of the desired form.
\end{proof}

\begin{proposition}\label{thm2:aut-L-Rtriv} 
Let $A$ be a cyclic Albert division algebra over a field $k$ and $L\subset A$ be a cyclic cubic subfield. There there exists $v\in L^{\times}$ such that $A^{(v)}$ contains a Tits process algebra $J(LK,*,1,\nu)$ and $\text{Aut}(A^{(v)}/L)\subset R{\text{\bf Str}}(A)(k)$, here $A^{(v)}$ denotes the $v$-isotope of $A$.
\end{proposition}
\begin{proof} By Lemma \ref{special-subalg}, there exists $v\in L^{\times}$ such that $A^{(v)}$ contains a subalgebra isomorphic to a Tits process $J(LK,*,1,\nu)$. We may therefore replace $A^{(v)}$ by an isomorphic Albert division algebra $A'$ that contains $S=J(LK,*,1,\nu)$ as a subalgebra.  

It therefore suffices to prove that $\text{Aut}(A'/L)\subset R\text{\bf Str}(A')(k)$ to prove the assertion.
\vskip1mm
\noindent
{\bf Claim.} We have $\text{Aut}(A'/L)\subset R\text{\bf Str}(A')(k)$.
\vskip1mm
To prove this, let $\rho$ denote a generator of $Gal(L/k)$. By Lemma \ref{tits-pro-skol}, $\rho$ admits an extension $\widetilde{\rho}$ to an automorphism of $J(LK, *, 1, \nu)$, which in turn admits an extension to an automorphism of $A'$ by (Prop. 3.2, \cite{Th-2}), also denoted by $\widetilde{\rho}$. By Corollary \ref{Aut-R-triv-S}, the extension $\widetilde{\rho}$ belongs to $R\text{\bf Str}(A')(k)$. Note that $\widetilde{\rho}(L)=L$.
  
  Now let $\phi\in\text{Aut}(A'/L)$ be arbitrary. Then $\widetilde{\rho}^{-1}\phi\notin \text{Aut}(A'/L)$ since $\widetilde{\rho}$ restricts to $\rho$ on $L$. By Proposition \ref{fixedpoint}, $\widetilde{\rho}^{-1}\phi\in\text{Aut}(A', L)$ fixes a cubic subfield $M\subset A'$ pointwise and $M\neq L$. 

The subalgebra $S_0:=<L,M>$ of $A'$ generated by $L$ and $M$ is $9$-dimensional and is left invariant by $\widetilde{\rho}^{-1}\phi$. Hence, by Corollary \ref{Aut-R-triv-S} we have,
  $$\psi:=\widetilde{\rho}^{-1}\phi\in \text{Aut}(A', S_0)\subset R\text{\bf Str}(A')(k).$$
  Since $\widetilde{\rho}\in R\text{\bf Str}(A')(k)$, it follows that $\phi\in  R\text{\bf Str}(A')(k)$. We have thus shown that $\text{Aut}(A'/L)\subset R\text{\bf Str}(A')(k)$ and hence the claim is settled and the proof is complete.
\end{proof}
\begin{corollary}\label{aut-A-L-Rtriv} Let $A$ be a cyclic Albert division algebra over $k$ and $L\subset A$ a cyclic cubic subfield.
  There exists $v\in L^{\times}$ such that $\text{Aut}(A^{(v)},L)\subset R\text{\bf Str}(A)(k)$.
\end{corollary}
\begin{proof} Let $v\in L^{\times}$ be such that $\text{Aut}(A^{(v)}/L)\subset R\text{\bf Str}(A)(k)$, such $v$ exists by Proposition \ref{thm2:aut-L-Rtriv}. Let $\phi\in \text{Aut}(A^{(v)}, L)$. We may assume, by Proposition \ref{thm2:aut-L-Rtriv} that $\phi|L=\rho\neq 1$. Then $\widetilde{\rho}^{-1}\phi\in \text{Aut}(A^{(v)}/L)$ and $\widetilde{\rho}\in R\text{\bf Str}(A)(k)$, here $\widetilde{\rho}$ denotes an extension of $\rho$ to an automorphism of $A$ as constructed in the proof of Proposition \ref{thm2:aut-L-Rtriv} above. It follows that $\phi\in R\text{\bf Str}(A)(k)$. 
\end{proof}

\noindent
{\bf Main theorem.} We now proceed to prove the $R$-triviality of $\text{\bf Str}(A)$ for an Albert division algebra $A$ over a field $k$. 

The results proved in the above subsection are valid for \emph{cyclic} Albert division algebras. The following key result from (\cite{Th-3}) connects these with general Albert division algebras:
\begin{proposition}\label{cyclicity} Let $A$ be an Albert division algebra over a field $k$. Then there exists $v\in A^{\times}$ such that 
$A^{(v)}$, the $v$-isotope of $A$, is cyclic, i.e., $A^{(v)}$ contains a cyclic cubic subfield.
\end{proposition}

The following is the key to the proof of the Main theorem. Recall that for an extension $N/k$ of the base field $k$ and an algebra $S$ over $k$, we denote $S_N:=S\otimes_k N$. 
\begin{theorem}\label{Str-A-L-Rtriv} Let $A$ be a cyclic Albert division algebra over a field $k$ of arbitrary characteristic and $L\subset A$ a cyclic cubic subfield. Then $\text{Str}(A,L)\subset R\text{\bf Str}(A)(k)$.
\end{theorem}
\begin{proof} We make some reductions for the proof. We first note that the structure group of an Albert algebra is isotopy invariant, i.e., $\text{\bf Str}(A^{(w)})=\text{\bf Str}(A)$ for any invertible $w\in A$. We have therefore 
    $$\text{Str}(A^{(w)})=\text{Str}(A)~\text{and}~R\text{\bf Str}(A^{(w)})(k)=R\text{\bf Str}(A)(k),$$ 
    for any $w\in A$ invertible.

    In light of this fact, by passing to an isotope by a suitable element of $L^{\times}$, by (\S2 {\bf (b)}) and Lemma \ref{special-subalg}, we may assume that $A$ contains a Tits process subalgebra $S:=J(LK,*, 1, \nu)$ 

    Hence, by Proposition \ref{thm2:aut-L-Rtriv} and Corollary \ref{aut-A-L-Rtriv}, we have $\text{Aut}(A,L)\subset R\text{\bf Str}(A)(k)$.

    \noindent
    Now let $\psi\in\text{Str}(A,L)$ and $\rho$ be a nontrivial Galois automorphism of $L$ over $k$. We extend $\rho$ to an automorphism $\widetilde{\rho}$ of $S$ (see Lemma \ref{tits-pro-skol}) and further extend $\widetilde{\rho}$ to an automorphism of $A$, also denoted by $\widetilde{\rho}$ (see \cite{Th-2}, Prop. 3.2). Then $\widetilde{\rho}\in R\text{\bf Str}(A)(k)$ by Corollary \ref{Aut-R-triv-S}.

    Let $a:=\psi(1)$. Then $a\in L^{\times}$ and the norm similarity $R_{N(a)^{-1}}U_a\psi\in\text{Isom}(A)$. Since scalar homotheties and $U$-operators belong to $R\text{\bf Str}(A)(k)$ (see Lemma 4.1, \cite{Th-2})), we may assume, at the outset, that $\psi$ is an isometry of the norm on $A$. Then for $a:=\psi(1)\in L^{\times},~N(a)=1$.
    
    Since $\psi(L)=L$ and $\psi(1)=a$, by (\cite{J0}, Thm.1), $\psi_|L=R_a\gamma$, where $\gamma$ is an element of $Gal(L/k)$. If $\gamma=1$ we set $\psi':=\psi\widetilde{\rho}$ and $\psi'=\psi$ if $\gamma\neq 1$.

    Then $\psi'\in\text{Str}(A,L)$ and we have, when $\gamma=1$,
     $$\psi'^3(1)=(\psi\widetilde{\rho})(\psi\widetilde{\rho})(\psi\widetilde{\rho})(1)=a\rho(a)\rho^2(a)=N(a)=1.$$
     When $\gamma\neq 1$, we must have $\gamma=\rho$ or $\gamma=\rho^2$. Without loss of generality, assume $\gamma=\rho$. Since $\psi'=\psi$ in this case, we have
     $$(\psi')^3(1)=\psi^3(1)=(R_a\rho)^3(1)=a\rho(a)\rho^2(a)=N_L(a)=1.$$ 
     Hence, we have, in either case, $\psi'^3\in\text{Aut}(A,L)\subset R\text{\bf Str})(A)(k)$. But $\widetilde{\rho}\in R\text{\bf Str}(A)(k)$, hence, when $\gamma=1$, we see that
     $(\psi\widetilde{\rho})(\psi\widetilde{\rho})\psi\in R\text{\bf Str}(A)(k)$. Conjugating this successively by $\psi^{-1}$, using normality of $R\text{\bf Str}(A)(k)$ in $\text{Str}(A)$ and using the fact that $\widetilde{\rho}\in R\text{\bf Str}(A)(k)$, it follows that $\psi^3\in R\text{\bf Str}(A)(k)$. In the case $\gamma\neq 1$, $\psi^3\in\text{Aut}(A,L)\subset R\text{\bf Str}(A)(k)$. Hence, we always have $\psi^3\in R\text{\bf Str}(A)(k)$. 

    Now, if $\gamma=1$, the norm similarity $\eta:=U_{a^{-1}}\psi^2$ is an automorphism of $A$ and $\eta(L)=L$. Thus $\eta\in R\text{\bf Str}(A)(k)$. Since $U$-operators belong to $R\text{\bf Str}(A)(k)$(see Lemma 4.1, \cite{Th-2}), we see that $\psi^2\in R\text{\bf Str}(A)(k)$.

    If $\gamma\neq 1$, then $\gamma=\rho$ or $\gamma=\rho^2$. We assume $\gamma=\rho$ without loss of generality. Then, since $\psi_|L=R_a\gamma$, the norm similarity $\eta:=\psi\widetilde{\rho}^{-1}$ of $A$ satisfies $\eta(L)=L$ and $\eta_|L=R_a$ and we are in the earlier case. Hence we conclude that $\eta^2\in R\text{\bf Str}(A)(k)$ and hence, arguing as above, we get $\psi^2\in R\text{\bf Str}(A)(k)$. It therefore follows that $\psi\in R\text{\bf Str}(A)(k)$. So we have proved the result.
  \end{proof}

\begin{corollary}\label{aut-A-Rtriv} Let $A$ be a cyclic Albert division algebra over $k$ and $L\subset A$ a cyclic cubic subfield. Let $v\in L^{\times}$ be such that $\text{Aut}(A^{(v)},L)\subset R\text{\bf Str}(A)(k)$. Then $\text{Aut}(A^{(v)})\subset R\text{\bf Str}(A)(k)$.
\end{corollary}
\begin{proof} Let $\phi\in\text{Aut}(A^{(v)})$. We may assume $\phi(L)\neq L$. Then $J:=<L,\phi(L)>$ is a $9$-dimensional subalgebra of $A^{(v)}$. Consider the embeddings $\phi:L\rightarrow J$ and the inclusion $\iota:L\rightarrow J$. Then, by (\cite{GP}, Thm. 5.2.7), there exists $\theta\in \text{Str}(J)$
        such that $\theta\phi=\iota R_w$ for some element $w\in L$ with $N_L(w)=1$. Then we have, for any $x\in L$,
        $ \theta\phi(x)=\iota R_w(x)$. Hence $\theta(\phi(x))=wx\in L~\forall x\in L$. Therefore $\theta\phi(L)=L$.

        By (Thm. 4.4, \cite{Th-2}), we can extend $\theta$ to a norm similarity of $A$,
        which we continue to denote by $\theta$. Then $\theta\in \text{Str}(A,J)$ and hence $\theta\in R\text{\bf Str}(A)(k)$ by 
Proposition \ref{Str-R-triv}. Also $\theta\phi\in\text{Str}(A,L)\subset R\text{\bf Str}(A)(k)$, by Theorem \ref{Str-A-L-Rtriv}. Therefore $\phi\in R\text{\bf Str}(A)(k)$. 
\end{proof}
The follwoing computation will be needed:
\begin{lemma}\label{computational} Let $A$ be a Jordan algebra over a field $k$ and $\psi\in\text{Str}(A)$. Then, for $v:=\psi(1)\in A^{\times}$ we have $\text{Aut}(A^{(v^{-1})})=\psi\text{Aut}(A)\psi^{-1}$.
\end{lemma}
\begin{proof} Clearly $\psi f\psi^{-1}$ is a norm similarity for $A^{(v^{-1})}$. Thus, to prove the assertion, it suffices to check that for $f\in\text{Aut}(A)$, $\psi f\psi^{-1}$ is an automorphism of $A^{(v^{-1})}$ and hence it suffices to check $\psi f\psi^{-1}$ fixes the identity element $v=\psi(1)$ of $A^{(v^{-1})}$. We have,
  $$\psi f\psi^{-1}(v)=\psi f\psi^{-1}(\psi(1))=\psi f(1)=\psi(1)=v.$$
  This proves the assertion. 
  \end{proof}
\begin{theorem}\label{Str-A-M-Rtriv} Let $A$ be an Albert division algebra over $k$. Then there exists $v\in A$ such that for any separable cubic subfield $M\subset A^{(v)}$, we have
  $\text{Str}(A^{(v)},M)\subset R\text{\bf Str}(A)(k)$.
  \end{theorem}
  \begin{proof} To begin with, since the structure group is unchanged by moving to an isotope, and taking istopes is transitive, we may replace $A$ by a suitable isotope and assume that $A$ contains a cyclic cubic subfield $L$ (see Proposition \ref{cyclicity}).

    By Theorem \ref{Str-A-L-Rtriv} we have $\text{Str}(A,L)\subset R\text{\bf Str}(A)(k)$. Also, by Corollary \ref{aut-A-L-Rtriv}, there exists $v\in L^{\times}$ such that
    $\text{Aut}(A^{(v)},L)\subset R\text{\bf Str}(A)(k)$. For this choice of $v\in L^{\times}$, by Corollary \ref{aut-A-Rtriv}, we have $\text{Aut}(A^{(v)})\subset R\text{\bf Str}(A)(k)$. 

    Let $M\subset A^{(v)}$ be any separable cubic subfield. Since $\text{Str}(A^{(v)},L)=\text{Str}(A,L)$, by the above, we may assume that $M\neq L$ and Theorem \ref{Str-A-L-Rtriv} allows us to assume  that $M/k$ is not cyclic.

    For simplicity of notation,  we denote $A_1=A^{(v)}$ and write $1$ for the identity element of $A_1$. Let $S:=<L,M>$ denote the subalgebra of $A_1$ generated by $L$ and $M$. Then $S$ is $9$-dimensional and contains $M$ as a subfield.

  Let $\psi\in\text{Str}(A_1,M)$. Then, by Proposition \ref{action} and notation introduced there, $[\psi(S)]$ is a subalgebra of $A_1$, where $[\psi(S)]=\psi(S)$ as a vector subspace of $A_1$, but the identity element of $[\psi(S)]$ equals that of $A_1$, which we denote by $1$.

  Now, $\psi\in\text{Str}(A_1)=\text{Str}(A)$, hence $\psi:A_1\rightarrow A_1^{(a^{-1})}$ is an isomorphism, with $a:=\psi(1)\in M\subset A_1$
  (see \S2 {\bf(c)}). Therefore $\psi: S\rightarrow\psi(S)$ is an isomorphism of $S$ onto the subalgebra $\psi(S)$ of $\psi(A_1)=A_1^{(a^{-1})}$ (see \S2 {\bf(c)}). Similarly $\psi: [\psi(S)]\rightarrow \psi([\psi(S)])$ is an isomorphism onto the subalgebra $\psi([\psi(S)])$ of $A_1^{(a^{-1})}$.

  Now, since $a\in M\subset S$, we have $U_a\in\text{Str}(A_1,S)$. Moreover, since $M$ is not cyclic, by (\cite{J0}, Thm.1), $\psi_|M=R_a$. Hence
    $\psi^2U_{a^{-1}}\in\text{Aut}(A_1)$. We have therefore, writing $\phi:=\psi^2U_{a^{-1}}$,
    $$\psi([\psi(S)])=\psi(\psi(U_{a^{-1}}(S))=\phi(S)\cong S.$$
    Hence $\psi([\psi(S)])$ is a subalgebra of $A_1^{(a^{-1})}$ isomorphic to $S$. On the other hand $S\cong\psi(S)\subset A_1^{(a^{-1})}$. Hence $\psi([\psi(S)])\cong \psi(S)$ as subalgebras of $A_1^{(a^{-1})}$. Let $\theta:\psi([\psi(S)])\rightarrow\psi(S)$ be an isomorphism. By (Thm. 3.1, \cite{P-S-T1}, Thm. 5.2, Remark 5.6, \cite{P}), we can extend $\theta$ to an element of $\text{Aut}(A_1^{(a^{-1})})$, also denoted by $\theta$. Hence we have 
    $$\theta(\psi([\psi(S)])=\theta\psi^2(S)=\theta\psi^2U_{a^{-1}}(S)=\theta\phi(S)=\psi(S).$$
    Now, by Lemma \ref{computational}, $\text{Aut}(A_1^{(a^{-1})})=\psi\text{Aut}(A_1)\psi^{-1}\subset\text{Str}(A)$ and $\text{Aut}(A_1)=\text{Aut}(A^{(v)})\subset R\text{\bf Str}(A)(k)$. Hence it follows that
    $\text{Aut}(A_1^{(a^{-1})})\subset R\text{\bf Str}(A)(k)$. Thus $\theta\in  R\text{\bf Str}(A)(k)$. We have $\phi\in\text{Aut}(A_1)\subset R\text{\bf Str}(A)(k)$ and, by (Lemma 4.1, \cite{Th-2}), $U$-operators (of isotopes) are in $R\text{\bf Str}(A)(k)$.

    Finally, from the above, we have $\psi^{-1}\theta\phi(S)=S$. Thus
    $$\psi^{-1}\theta\phi\in\text{Str}(A_1,S)\subset R\text{\bf Str}(A_1)(k)=R\text{\bf Str}(A)(k),$$
    where the inclusion follows from Proposition \ref{Str-R-triv}. Hence it follows that $\psi\in\text{\bf Str}(A)(k)$. 
    \end{proof}
    \noindent
    {\bf Remark:} The above proof gives a bit more. Let $A$ be an Albert division algebra over a field $k$ and $A_0$ be an isotope of $A$ that contains a cubic cyclic subfield $L$. Let $v\in L^{\times}$ be such that $\text{Aut}(A_0^{(v)},L)\subset R\text{\bf Str}(A)(k)$ and $\text{Aut}(A_0^{(v)})\subset R\text{\bf Str}(A)(k)$. Then for any separable cubic subfield $M\subset A_0^{(v)}$ we have
    $\text{Str}(A_0^{(v)},M)\subset R\text{\bf Str}(A)(k)$.

    \noindent
  We now prove our {\bf Main theorem}. 
  \begin{theorem}\label{main} Let $A$ be an Albert algebra over a field $k$ of arbitrary characteristic. Then $\text{\bf Str}(A)$ is $R$-trivial over $k$.
  \end{theorem}
  \begin{proof} When $A$ is reduced, this is Theorem 7.3 of \cite{Th-2}. So we may assume $A$ is a division algebra. We first prove $\text{Str}(A)=R\text{\bf Str}(A)(k)$.

    By Proposition \ref{cyclicity}, Corollary \ref{aut-A-L-Rtriv} and Corollary \ref{aut-A-Rtriv} we may assume, by passing to a suitable isotope, that $A$ contains a cubic cyclic subfield $L$, $\text{Aut}(A)\subset R\text{\bf Str}(A)(k)$ and that $\text{Str}(A,M)\subset R\text{\bf Str}(A)(k)$ for any separable cubic subfield $M\subset A$ (see Theorem \ref{Str-A-M-Rtriv} and its proof, the Remark above).
    
    Let $\psi\in\text{Str}(A)$ be arbitrary. By Proposition \ref{invariant-cubic}, there exists $M\subset A$, a cubic subfield, and $\theta\in R\text{\bf Str}(A)(k)$ such that $\theta\psi\in\text{Str}(A,M)$. When $M$ is separable, Theorem \ref{Str-A-M-Rtriv} implies $\theta\psi\in R\text{\bf Str}(A)(k)$ and thus $\psi\in R\text{\bf Str}(A)(k)$.

    So assume now $M/k$ is not separable. Then necessarily the characteristic of $k$ is $3$. Hence for $a:=\psi(1)\in M^{\times}$ we have $a^3=\alpha$ for some $\alpha\in k^{\times}$ and $R_{{\alpha}^{-1}}U_a\psi\in\text{Aut}(A)\subset R\text{\bf Str}(A)(k)$. Since scalar homotheties and $U$-operators belong to $R\text{\bf Str}(A)(k)$, it follows that $\psi\in R\text{\bf Str}(A)(k)$.

    Now let $F/k$ be any field extension. If $A_F:=A\otimes_kF$ is reduced, we have, by Theorem 7.3 of \cite{Th-2}, $\text{Str}(A_F)=R\text{\bf Str}(A_F)(F)$. If $A_F$ is division algebra over $F$, the equality $\text{Str}(A_F)=R\text{\bf Str}(A_F)(F)$ follows from the above argument by replacing $k$ by $F$ and $A$ by $A_F$ . This completes the proof of the theorem.
\end{proof} 
\begin{corollary}\label{isom-R-triv} Let $A$ be an Albert division algebra over a field $k$. Then the algebraic group $\text{\bf Isom}(A)$, the full group of norm isometries of $A$, is $R$-trivial.
  \end{corollary}
  \begin{proof} The proof is exactly the same as (Thm. 5.5, \cite{Th-2}), so we omit it here.
    \end{proof}

\section{\bf $R$-triviality of $E^{78}_{8,2}$, Tits-Weiss conjecture} In this section, we prove the Kneser-Tits conjecture for algebraic groups over $k$ with Tits index $E^{78}_{8,2}$ as well as the Tits-Weiss conjecture for Albert division algebras. We recall a few facts below.

\vskip1mm
\noindent
Let $\text{\bf G}$ be a simple, simply connected algebraic group with Tits index $E_{8,2}^{78}$ over $k$. 
This index has associated Dynkin diagram 
\vskip10mm

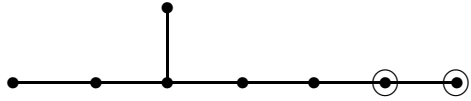
\begin{center}
\setlength{\unitlength}{1.0cm}
\begin{picture}(10,-25)\thicklines

\put(0,0){\line(1,0){1}}
\put(1,0){\line(1,0){1}}
\put(2.10,0){\line(0,1){1}}
\put(2,0){\line(1,0){1}}
\put(3,0){\line(1,0){1}}
\put(4,0){\line(1,0){1}}
\put(-0.05,-0.10){$\bullet$}
\put(1.05,-0.10){$\bullet$}
\put(2.00,-0.10){$\bullet$}
\put(3.0,-0.10){$\bullet$}
\put(3.95,-0.10){$\bullet$}
\put(2.00,0.90){$\bullet$}
\put(4.90,-0.10){$\bullet$}
\put(4.92,0){\line(1,0){1}}
\put(4.8,-0.10){$\bigcirc$}
\put(5.85,-0.10){$\bullet$}
\put(5.74,-0.10){$\bigcirc$}
\end{picture}
\end{center}
The reductive anisotropic kernel of $\text{\bf G}$ is simple, simply connected and strongly inner of type $E_6$ (see \cite{TW} Appendix, \cite{T1}) and the semisimple anisotropic kernel of $\text{\bf G}$ is isomorphic to the group {\bf Isom}$(A)$ of norm isometries  of an Albert 
division algebra $A$ over $k$ (see \cite{T1}). Let $\Gamma$ be the Tits building associated to $\text{\bf G}$. Then $\Gamma$ is a Moufang hexagon associated to a hexagonal system of type $27/K$, where $K$ is either a quadratic field extension of $k$ or $K=k$, according as $A$ is represented as the second Tits construction or the first respectively. The group $\text{\bf G}(k)$ is then the group of ``linear'' automorphisms of $\Gamma$. 
Any group of type $E_{8,2}^{78}$ can be realized as the group of linear automorphisms of a Moufang hexagon as above (see \cite{TW}).

\vskip1mm
\noindent 
We continue to use notation set up in the earlier sections. For an Albert algebra $A$ over a field $k$, one knows that the structure group $\text{\bf Str}(A)$ is connected, reductive of type $E_6$ (see Thm. 7.3.2, \cite{SV}. Thm. 14.27, \cite{Spr-J}). The group $\text{\bf Isom}(A)$ of norm isometries of $A$, by (\cite{SV}, Thm. 7.3.2), is a connected, simple, simply connected algebraic group of type $E_6$ defined over $k$. Moreover, $\text{\bf Isom}(A)=[\text{\bf Str}(A),\text{\bf Str}(A)]$, the commutator subgroup of $\text{\bf Str}(A)$. Let $\text{Instr}(A)$
denote the subgroup of $\text{Str}(A)$ generated by the $U$-operators of $A$. Finally, let $C$ be the subgroup of all scalar homotheties $R_a,~a\in k^{\times}$ in $\text{Str}(A)$. We can now state  
\vskip1mm
\noindent
{\bf Tits-Weiss conjecture.} Let $A$ be an Albert division algebra over a field $k$, then $\text{Str}(A)=C.Instr(A)$. 
\vskip0.5mm
\noindent
This is equivalent to proving $\text{Str}(A)/(C.\text{Instr}(A))=\{1\}$. The Tits-Weiss conjecture is equivalent to the \emph{Kneser-Tits conjecture} for groups of type $E^{78}_{8,2}$ (see Appendix of \cite{ACP} for a proof). 

Let $A$ be an Albert division algebra over a field $k$. To (the isotopy class of) $A$, one can attach a simple, simply connected group $\text{\bf G}$ defined over $k$, having Tits index $E^{78}_{8,2}$. Let $\text{\bf G}(k)^{\dagger}$ denote the subgroup of $\text{\bf G}(k)$ generated by the $k$-rational points of the unipotent radicals of $k$-parabolic subgroups of $\text{\bf G}$. The Kneser-Tits conjecture says $\text{\bf G}(k)/\text{\bf G}(k)^{\dagger}=\{1\}$. 

Conversely, for a simple, simply connected algebraic group $\text{\bf G}$ defined over $k$ having Tits index $E^{78}_{8,2}$, the (reductive) anisotropic kernel is the structure group $\text{\bf Str}(A)$ of an Albert division algebra $A$ defined over $k$, whose isotopy class is determined by the isomorphism class of $\text{\bf G}$ (see \cite{T1}, 3.3.1). In (\cite{ACP}, Appendix) it is shown that 
$\text{\bf G}(k)/\text{\bf G}(k)^{\dagger}\cong\text{Str}(A)/(C.\text{Instr}(A))$. Hence, to prove the Tits-Weiss conjecture,
it suffices to prove the Kneser-Tits conjecture for $\text{\bf G}$. By (Thm. 7.2, \cite{G}), this is equivalent to proving $\text{\bf G}$ is $R$-trivial. Hence it suffices to prove
$$W(k,\text{\bf G})=\text{\bf G}(k)/\text{\bf G}(k)^{\dagger}\cong \text{\bf G}(k)/R=\{1\}.$$

Fix a maximal $k$-split torus $\text{\bf S}\subset\text{\bf G}$. It is shown in (\cite{Th-2}, 6.1) that 
$\text{\bf G}(k)/R\cong Z_{\text{\bf G}}(\text{\bf S})(k)/R\cong (Z_{\text{\bf G}}(\text{\bf S})/\text{\bf S})(k)/R$ (see \cite{G2}, 1.2).

These isomorphisms are compatible with base change by any field extension $L/k$. 
To prove the Kneser-Tits conjecture for $\text{\bf G}$, it therefore suffices to prove $(Z_{\text{\bf G}}(\text{\bf S})/\text{\bf S})(k)/R=\{1\}$.

Let $\text{\bf H}:=Z_{\text{\bf G}}(\text{\bf S})$ and $\text{\bf H}':=[\text{\bf H}, \text{\bf H}]$ be the commutator subgroup of $\text{\bf H}$. Then $\text{\bf H}'$ is simple, simply connected, $k$-anisotropic and strongly inner of type $E_6$. Since $\text{\bf H}$ is the reductive anisotropic kernel of {\bf G}, by (\cite{T1}, 3.3), we have $\text{\bf H}'=\text{\bf Isom}(A)$. 
\begin{lemma}[Lemma 6.1, \cite{Th-2}] The connected center of {\bf H} equals $S$.
\end{lemma}
\noindent
From this it follows that 
$$(Z_{\text{\bf G}}(\text{\bf S})/\text{\bf S})(k)/R=(\text{\bf H}'\text{\bf S}/\text{\bf S})(k)/R\cong(\text{\bf H}'/\text{\bf H}'\cap
\text{\bf S})(k)/R\cong(\text{\bf Str}(A)/\text{\bf Z})(k),$$
where $\text{\bf Z}$ is the center of $\text{\bf Str}(A)$ (see \cite{Th-2}, 6.1). Recall that $\text{\bf Z}\cong\mathbb{G}_m$ over $k$ (see \cite{Th-2}, Lemma 6.2). 
Hence, by Hilbert Theorem-$90$, $H^1(k,\text{\bf Z})=\{1\}$. It follows that $(\text{\bf Str}(A)/\text{\bf Z})(k)\cong\text{\bf Str}(A)(k)/\text{\bf Z}(k)$ .

\noindent
We now prove
\begin{theorem}\label{Weiss} Let $\text{\bf G}$ be a simple, simply connected algebraic group over $k$ with Tits index $E^{78}_{8,2}$. Then {\bf G} is $R$-trivial.   
\end{theorem}
\begin{proof} Let {\bf H} be the reductive anisotropic kernel of $\text{\bf G}$. Then $[\text{\bf H},\text{\bf H}]={\bf Isom}(A)$ for an Albert division algebra $A$ over $k$. We have, by the above discussion, 
  $$\text{\bf G}(k)/R\cong(\text{\bf Str}(A)/\text{\bf Z})(k)/R\cong(\text{\bf Str}(A)(k)/\text{\bf Z}(k))/R.$$
 
\noindent
Let $L$ be any field extension of $k$. If $A\otimes_k L$ is division,
then the Tits index of $\text{\bf G}$ over $L$ remains the same and the anisotropic kernel of $\text{\bf G}$ over $L$ corresponds to {\bf Str}$(A\otimes_kL)$. Hence 
$$W(L,\text{\bf G})= \text{\bf G}(L)/R=(\text{\bf Str}(A)(L)/\text{\bf Z}(L))/R=\{1\},$$
the last equality follows as $\text{\bf Str}(A)$ is $R$-trivial (see Theorem \ref{main}).
When $A\otimes_k L$ is reduced, the result follows from (Thm. 7.3 and Thm. 7.5, \cite{Th-2}). This proves the theorem.
\end{proof}
\begin{corollary}\label{KnTits-1} Let $\text{\bf G}$ be a simple, simply connected algebraic group over $k$ with Tits index $E^{78}_{8,2}$. The the Kneser-Tits conjecture holds for $\text{\bf G}$. 
\end{corollary}
\begin{proof} Let $\text{\bf G}$ be a simple, simply connected algebraic group over $k$ with Tits index $E^{78}_{8,2}$. By the above $R$-triviality result we have $\text{\bf G}(k)/R=\{1\}$ and by (Thm. 7.2 of \cite{G}) we have $\text{\bf G}(k)/\text{\bf G}(k)^{\dagger}\cong\text{\bf G}(k)/R$. By $R$-triviality of $\text{\bf G}$ proved above, the latter quotient is trivial. Hence the result follows.
\end{proof}

\vskip1mm
\noindent
    {\bf Remark.} Note that we have proved above also that $\text{\bf Str}(A)/\text{\bf Z}$, the adjoint form of $E_6$, corresponding to an Albert division algebra $A$ over a field $k$ is $R$-trivial. 
\begin{corollary}{\bf (Tits-Weiss Conjecture)}\label{TW} Let $A$ be an Albert division algebra over a field $k$. Then $\text{Str}(A)=C.\text{Instr}(A)$.
\end{corollary}
\begin{proof} This is immediate from the above and the equivalence of the Tits-Weiss conjecture and the Kneser-Tits conjecture for algebraic groups with Tits index $E_{8,2}^{78}$ proved in (Appendix, \cite{ACP}). Given an Albert division algebra $A$ over $k$, let $\text{\bf G}$ be the group of type $E_{8,2}^{78}$ associated to $A$. The group of $k$-points $\text{\bf G}(k)$ occurs as the group of ``linear'' automorphisms of the Moufang hexagon constructed from $A$. In (Appendix, \cite{ACP}), it is shown that the quotient 
$\text{Str}(A)/(C.\text{Instr}(A))$ is isomorphic to the quotient $\text{\bf G}(k)/\text{\bf G}(k)^{\dagger}$. By the above, this quotient is trivial. Hence we have proved the Tits-Weiss conjecture. 
\end{proof}
\section{\bf $R$-triviality of $E^{78}_{7,1}$}
We now prove $R$-triviality for simple groups of type $E^{78}_{7,1}$. These groups also have anisotropic kernels 
as the structure groups of Albert division algebras and all groups of this type arise this way (see \cite{T1}, 3.3.1). 

\vskip1mm
\noindent
Let ${\bf \Xi}$ be a simple, simply connected algebraic group over $k$ with Tits index $E_{7,1}^{78}$ over $k$. 
Then the semisimple anisotropic kernel of ${\bf \Xi}$ is a simple, simply connected strongly inner group of type $E_6$ and is isomorphic to {\bf Isom}$(A)$ for an Albert division algebra $A$ over $k$ (see \cite{T1}). The Tits index of ${\bf \Xi}$ has the diagram: 

\vskip10mm

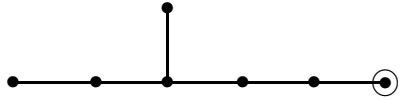
\begin{center}
\setlength{\unitlength}{1.0cm}
\begin{picture}(10,-25)\thicklines

\put(0,0){\line(1,0){1}}
\put(1,0){\line(1,0){1}}
\put(2.10,0){\line(0,1){1}}
\put(2,0){\line(1,0){1}}
\put(3,0){\line(1,0){1}}
\put(4,0){\line(1,0){1}}
\put(-0.05,-0.09){$\bullet$}
\put(1.05,-0.09){$\bullet$}
\put(2.00,-0.09){$\bullet$}
\put(3.0,-0.09){$\bullet$}
\put(3.95,-0.09){$\bullet$}
\put(2.00,0.90){$\bullet$}
\put(4.90,-0.10){$\bullet$}
\put(4.80,-0.10){$\bigcirc$}

\end{picture}
\end{center}

We can realize ${\bf \Xi}$ as follows (see \cite{Ko}): let $A$ be an Albert division algebra over $k$. Let ${\bf \Xi}(A)$ be the group of birational transformations of $A$ 
generated by the structure group ${\bf Str}(A)$ and the maps 
$$x\mapsto t_a(x)=x+a~(~a\in~A_{\overline{k}}~),~x\mapsto j(x)=-x^{-1}~(x\in~A_{\overline{k}}~).$$
Then ${\bf \Xi}(A)$ has type $E_{7,1}^{78}$ and any algebraic group over $k$ with this Tits index arises this way (see \cite{T1}). Any element of $f\in{\bf \Xi}(A)$ is of the form 
$$f=w\circ t_a\circ j\circ t_b\circ j\circ t_c,~~w\in\text{\bf Str}(A),~a,b,c\in~A_{\overline{k}}.$$ 
This was proved by Koecher when $k$ is infinite and has characteristic $\neq 2$ and by O. Loos in arbitrary characteristics (\cite{L}), in the context of Jordan pairs.

The groups ${\bf \Xi}$ with this index have anisotropic kernel the structure group of an Albert division algebra. We use the explicit description of ${\bf \Xi}$ as above. 

Let $A$ and ${\bf \Xi}(A)$ be as above. Clearly $\text{\bf Str}(A)\subset {\bf \Xi}(A)$. Let {\bf Z} be the center of $\text{\bf Str}(A)$, then $\text{\bf Z}\cong\mathbb{G}_m$ over $k$. 
Fix $\text{\bf S}\subset {\bf \Xi}(A)$, a maximal $k$-split torus containing {\bf Z}. Then $\text{\bf H}:=Z_{{\bf \Xi}}(\text{\bf S})$ is the reductive anisotropic kernel of ${\bf \Xi}={\bf \Xi}(A)$.
We have $\text{\bf S}=\text{\bf Z}$, since $k$-rank of ${\bf \Xi}$ is $1$. Therefore $Z(\text{\bf H})^{\circ}=\text{\bf Z}$ and it follows (see \cite{Th-2}, \S7) that 
$${\bf\Xi}(k)/R\cong Z_{{\bf \Xi}}(\text{\bf S})(k)/R\cong (Z_{{\bf \Xi}}(\text{\bf S})/\text{\bf S})(k)/R.$$
\noindent
Also, 
$$\text{\bf H}=Z_{{\bf \Xi}}(\text{\bf S})=\text{\bf H}'.Z(\text{\bf H})^{\circ}=\text{\bf Isom}(A).\text{\bf Z}=\text{\bf Str}(A).$$
\noindent
Hence 
$$Z_{{\bf \Xi}}(\text{\bf S})/\text{\bf S}=\text{\bf Str}(A)/\text{\bf Z}.$$
We have proved that $(\text{\bf Str}(A)/\text{\bf Z})(k)/R=\{1\}$ (see Remark after the proof of Theorem \ref{Weiss}). Hence we have, by arguments quite similar to the proof of $R$-triviality of groups with index $E^{78}_{8,2}$, the following $R$-triviality result:
\begin{theorem}\label{Koecher} Let ${\bf \Xi}$ be a simple, simply connected algebraic group over $k$ with Tits index $E^{78}_{7,1}$. Then ${\bf \Xi}$ is $R$-trivial. 
\end{theorem} 
\begin{corollary}\label{KnTits-2} Let ${\bf \Xi}$ be a simple, simply connected algebraic group over $k$ with Tits index $E^{78}_{7,1}$. Then the Kneser-Tits conjecture holds for ${\bf \Xi}$. 
\end{corollary}
\vskip1mm
\noindent
    {\bf Acknowledgement.} This work was partially funded by the DFG under Germany's Excellence
Strategy ``EXC 2044-390685587, Mathematics M\"unster: Dynamics-Geometry-Structure'', when the author visited Math-Institute, M\"unster, in the summer of 2019. We thank Linus Kramer for his support. We thank Holger Petersson and Richard Weiss for their interest in our work. We thank colleagues B. Sury, Mathew Joseph and Siva Athreya for their support during the culmination of this work.

\vskip5mm
Indian Statistical Institute, Stat-Math.-Unit, 8th Mile Mysore Road, Bangalore-560059, India.
\center email: maneesh.thakur@gmail.com

\end{document}